\documentclass[12pt]{amsart}
\usepackage{amsmath,amssymb}
\usepackage[utf8,utf8x]{inputenc}
\usepackage[pagebackref=true,
pdftitle = {Combinatorial\ simpliciality\ of\ arrangements\ of\ hyperplanes},
pdfkeywords = {arrangement\ of\ hyperplanes; simplicial; projective\ plane;},
pdfsubject = {52C30; 51E20; 20F55;},
pdfauthor = {}]
{hyperref}
\usepackage{hyperref}
\usepackage{longtable}
\usepackage{graphicx}
\usepackage{pict2e}

\newtheorem{propo}{Proposition}[section]
\newtheorem{corol}[propo]{Corollary}
\newtheorem{theor}[propo]{Theorem}

\newtheorem{conje}[propo]{Conjecture}

\newtheorem{oquest}[propo]{Open question}

\theoremstyle{definition}
\newtheorem{defin}[propo]{Definition}
\newtheorem{examp}[propo]{Example}

\theoremstyle{remark}
\newtheorem{remar}[propo]{Remark}

\numberwithin{equation}{section}

\newcommand{\NN }{\mathbb{N}}
\newcommand{\CC }{\mathbb{C}}
\newcommand{\RR }{\mathbb{R}}
\newcommand{\FF }{\mathbb{F}}
\newcommand{\QQ }{\mathbb{Q}}

\newcommand{\PP }{\mathbb{P}}

\newcommand{\ii }{\mathrm{i}}

\newcommand{\Rc }{\mathcal{R}}
\newcommand{\Ac }{\mathcal{A}}
\newcommand{\Ec }{\mathcal{E}}
\newcommand{\Bc }{\mathcal{B}}
\newcommand{\Dc }{\mathcal{D}}

\newcommand{\Cc }{\mathcal{C}}

\DeclareMathOperator{\diag}{diag}

\DeclareMathOperator{\Aut}{Aut}

\DeclareMathOperator{\GL}{GL}
\DeclareMathOperator{\PGL}{PGL}
\DeclareMathOperator{\PG}{PG}

\DeclareMathOperator{\md}{mod}

\newcommand{\Sc }{\mathcal{S}}
\newcommand{\Sym }{\mathbb{S}}
\newcommand{\PFq }{\PG(2,q)}

\newcommand{\lmultiset }{\{\hspace{-4.7pt}\{}
\newcommand{\rmultiset }{\}\hspace{-4.7pt}\}}

\title[Combinatorial simpliciality of arrangements]
{Combinatorial simpliciality of arrangements of hyperplanes}

\author{M.~Cuntz}
\address{Michael Cuntz,
Fachbereich Mathematik,
TU Kai\-sers\-lau\-tern,
Postfach 3049,
D-67653 Kaiserslautern, Germany}
\email{cuntz@mathematik.uni-kl.de}

\author{D.~Geis}
\address{David Geis,
Fachbereich Mathematik,
TU Kai\-sers\-lau\-tern,
Postfach 3049,
D-67653 Kaiserslautern, Germany}
\email{}

\begin{document}

\begin{abstract}
We introduce a combinatorial characterization of simpliciality for arrangements of hyperplanes. We then give a sharp upper bound for the number of hyperplanes of such an arrangement in the projective plane over a finite field, and present some series of arrangements related to the known arrangements in characteristic zero. We further enumerate simplicial arrangements with given symmetry groups. Finally, we determine all finite complex reflection groups affording combinatorially simplicial arrangements. It turns out that combinatorial simpliciality coincides with inductive freeness for finite complex reflection groups except for the Shephard-Todd group $G_{31}$.
\end{abstract}

\maketitle

\tableofcontents

\section{Introduction}

A simplicial arrangement is a central hyperplane arrangement $\Ac$ in $\RR^n$ such that every chamber is an open simplicial cone. Simplicial arrangements were first introduced and studied by Melchior \cite{a-Melchi41}. Subsequently Gr\"unbaum \cite{p-G-71} (see \cite{p-G-09} for a revised version and \cite{p-G-13} for the state of the art) collected the known simplicial arrangements in the real projective plane. Shortly afterwards, simplicial arrangements attracted attention in the seminal work of Deligne \cite{MR0422673}: They are a natural context to study the $K(\pi,1)$ property of complements of arrangements. Moreover, they appeared several times as examples or counterexamples to conjectures for instance about various freeness properties of arrangements.

For a long time Gr\"unbaum's catalogue seemed to be complete. The classification of crystallographic arrangements (see \cite{p-C10} and \cite{p-CH10}), a very large subclass of the rational simplicial arrangements, endorses this impression. However, recently the first author found four until then unknown simplicial arrangements \cite{p-C12}.

In this note we give some more evidence for the completeness of Gr\"unbaum's corrected catalogue by searching for simplicial arrangements using a completely different technique: Let $\Ac$ be a simplicial arrangement in $\RR^3$ and assume that the coordinates of the normal vectors of the hyperplanes may be chosen in a subfield $K\subseteq \RR$ which is a finite extension\footnote{Minimal fields of definition were computed in \cite{p-C10b} for all known simplicial arrangements. Most of the occurring fields are of the form $\QQ(\zeta)\cap\RR$ for $\zeta$ a root of unity.} of $\QQ$.
Choose an epimorphism $\rho : {\mathfrak{o}}_K \rightarrow \FF_q$ from the integers of $K$ onto a finite field $\FF_q$ for $q$ a prime power. The map $\rho$ induces a map on the arrangements in $K^3$ (after a suitable fixed choice of normal vectors in ${\mathfrak{o}}_K^3$). If $q$ is large enough, then the intersection lattices of $\Ac$ and of $\rho(\Ac)$ will be isomorphic
(see for instance \cite[Prop.\ 5.13]{rS-07}).

Now there is a priori no concept of simpliciality for arrangements in $\FF_q^3$, since the notion of an ``open simplicial cone'' does not make sense here. However, it turns out that a central essential arrangement $\Ac$ in $\RR^3$ is simplicial if and only if
\[ 3 (|P|-1) = \sum_{v\in P} |\{ H\in\Ac\mid v\subset H\}| \]
where $P$ is the set of one-dimensional elements in the intersection lattice $L(\Ac)$ of $\Ac$. Thus simpliciality is a combinatorial property of $L(\Ac)$ for arrangements in $\RR^3$, see also Def.\ \ref{def:combsimp} for a generalization to arbitrary dimensions using characteristic polynomials. But this property is then independent of the chosen field or of a notion of a chamber. In this note, we first observe the following simple but perhaps surprising formula.
\begin{theor}[see Thm.\ \ref{simpFq}]
Let $\Ac$ be an essential central arrangement in $\FF_q^3$. Then $\Ac$ is combinatorially simplicial if and only if
\[ |\Ac| = 3q-\frac{3n^\Ac_0+2n^\Ac_1}{q+1} \]
where $n^\Ac_i$ is the number of points in the projective plane $\PFq$ which lie on exactly $i$ hyperplanes of $\Ac$.
\end{theor}
In particular, we obtain that a simplicial arrangement in $\FF_q^3$ has at most $3q$ hyperplanes, and in fact this is a sharp bound, see Prop.\ \ref{Fq3q}.

The study of simplicial arrangements in $\FF_q^3$ has several advantages and provides some new insights about real and complex arrangements. The most useful benefit is that we can use counting arguments over $\FF_q$ since $\PFq$ is a finite set. Further, the symmetry group $\PGL_3(\FF_q)$ of the set of all arrangements is well studied. We collect many series of simplicial arrangements over $\FF_q$, some of them are related to complex reflection groups, others to the known real simplicial arrangements. But we also find some rare ``new'' complex arrangements satisfying the combinatorial simpliciality by realizing the incidences of the arrangements over $\FF_q$ using the algorithm described in \cite{p-C10b}.

The largest part of this note is devoted to the consideration of combinatorially simplicial arrangements in $\FF_q^3$ which have a given symmetry group. There are many symmetry groups associated to a real arrangement $\Ac$ in $\RR^3$, for instance the automorphism groups of the lattice $L(\Ac)$ or of the CW-complex induced by $\Ac$, or the subgroup of $\PGL_3(\RR)$ leaving $\Ac$ invariant.

To enumerate all simplicial arrangements in $\FF_q^3$ with given symmetry group $G\le\PGL_3(\FF_q)$, we first compute the possible embeddings $\iota : G\hookrightarrow \PGL_3(\FF_q)$ up to conjugacy by computing the Schur multiplier $M(G)=H^2(G,\CC^\times)$ of $G$ and inspecting the character table of a covering group.
We can then for each such $\iota$ build unions of orbits of the action of $\iota(G)$ on $\PFq$.
With this method we find many ``new'' simplicial arrangements, but in particular we also recover the incidence of the recently discovered simplicial arrangement with $25$ lines (see \cite{p-C12}).

In the last section of this note we determine the reflection arrangements which are combinatorially simplicial. Using \cite[Cor.\ 5.16]{p-BC10} and \cite[Thm.\ 1.1]{p-HR12} we observe a remarkable connection between our simpliciality and a combinatorially defined notion of freeness of the module of derivations of an arrangement called \emph{inductive freeness} (see \cite[Def.\ 4.53]{OT}):
\begin{corol}
The reflection arrangement of a finite irreducible complex reflection group not isomorphic to $G_{31}$ is combinatorially simplicial if and only if it is inductively free. The reflection arrangement of $G_{31}$ is simplicial but not inductively free.
\end{corol}

The enumerations we present here raise several questions which could shed light onto the yet mysterious family of simplicial arrangements:

\begin{oquest}
Are the newly found combinatorially simplicial complex arrangements $K(\pi,1)$?
\end{oquest}

\begin{oquest}
Are there more combinatorially simplicial arrangements in the complex projective plane with nontrivial symmetry group?
\end{oquest}

\begin{oquest}
Is there a general connection between combinatorial simpliciality and inductive freeness of arrangements?
\end{oquest}

We further propose an extension of Gr\"unbaum's Conjecture \cite{p-G-09}:

\begin{conje}
There exists an $N\in\NN$ such that
the incidence of a simplicial arrangement in $\CC^3$ with more than $N$ hyperplanes is isomorphic to the incidence of
$\Ac(2n,1)$, $\Ac(4n+1,1)$, or to the incidence of the imprimitive reflection groups $G(e,1,3)$ for some $e,n\in\NN$.
\end{conje}

\textbf{Acknowledgement.}
We would like to thank U.\ Dempwolff and G.\ Malle for many valuable comments.
The definition of combinatorial simpliciality in dimension greater than three was suggested by M.\ Falk.

\section{Simpliciality of arrangements by counting cells}

An \emph{arrangement of hyperplanes} $\Ac$ is a finite set of hyperplanes in a vector space $V$ of dimension $r$. The \emph{intersection lattice} $L(\Ac)$ of $\Ac$ is the set of all intersections of elements of $\Ac$ ordered by reverse inclusion.
Throughout this note all arrangements are assumed to be \emph{central}, i.e.\ the hyperplanes are linear subspaces. We further assume all arrangements to be \emph{essential}, i.e.\ the rank of a maximal element of $L(\Ac)$ is $\dim V = r$. For a central arrangement this means that the intersection of all hyperplanes is $0$.

An arrangement $\Ac$ of hyperplanes in $\RR^r$, $r\ge 2$ is called \emph{simplicial}, if all \emph{chambers} of $\Ac$ (the connected components of $\RR^r\backslash\cup_{H\in\Ac} H$) are open simplicial cones.
Equivalently, $\Ac$ induces a simplicial cell decomposition of the sphere $\Sc^{r-1}$ which we denote $\Cc(\Ac)$.
Let $c_n$ denote the number of $n$-cells of $\Cc(\Ac)$.
Notice that $c_{r-1}$ is the number of chambers and $c_{r-2}$ the number of walls. If $\Ac$ is simplicial, then every wall is adjacent to two chambers, and every chamber is adjacent to $r$ walls, thus
$r c_{r-1} = 2 c_{r-2}$. It turns out that this condition is sufficient for simpliciality of a central essential arrangement:

\begin{propo}
Let $\Ac$ be a central essential arrangement of hyperplanes in $\RR^r$, $r\ge 2$. Then $\Ac$ is simplicial if and only if
\[ r c_{r-1} = 2 c_{r-2}. \]
\end{propo}
\begin{proof}
Assume that $\Ac$ is an arrangement with $r c_{r-1} = 2 c_{r-2}$.
Let $K$ be a chamber (equivalently an $r-1$-cell). Since $K$ is a connected component of $\RR^r\backslash\cup_{H\in\Ac} H$ and $\Ac$ is essential, there are at least $r$ $0$-cells adjacent to $K$. Further, each $0$-cell is adjacent to a wall of $K$ and at most $r-1$ $0$-cells may be adjacent to one same wall. Thus every chamber has at least $r$ walls.
Assume that $\Ac$ has a chamber with more than $r$ walls. Then
\[ 2 c_{r-2} = \sum_{C \text{ chamber}} |\{(r-2)\text{-cells adjacent to }C\}| > r c_{r-1} = 2 c_{r-2} \]
which is a contradiction.
\end{proof}

\begin{corol}\label{combsimp_1}
Let $\Ac$ be a central essential arrangement of hyperplanes in $\RR^r$. Then $\Ac$ is simplicial if and only if
\[ \sum_{n=0}^{r-3} (-1)^n c_n = 1+(-1)^{r-1} + (-1)^{r-1}\frac{r-2}{r} c_{r-2}. \]
In particular, if $r=3$, then $\Ac$ is simplicial if and only if
\[ 3 c_0 = 6 + c_1. \]
\end{corol}
\begin{proof}
Since the Euler characteristic of $\Sc^{r-1}$ is $1+(-1)^{r-1}$, we have
$\sum_n (-1)^n c_n = 1+(-1)^{r-1}$.
\end{proof}

\begin{defin}[{\cite[1.13, 2.25]{OT}}]
Let $\Ac$ be an arrangement in $V$.
For $X\in L(\Ac)$, we call the arrangement
\[ \Ac^X=\{X\cap H\mid H\in\Ac,\:\: X\not\subseteq H, \mbox{ and } X\cap H\ne \emptyset\}\]
in $X$ the \emph{restriction} of $\Ac$ to $X$.\\
The \emph{characteristic polynomial} $\chi_\Ac$ of $\Ac$ is defined by
\[ \chi_\Ac(t) = \sum_{X\in L(\Ac)}\mu(X) t^{\dim(X)}, \]
where $\mu$ is the M\"obius function of the lattice $L(\Ac)$.
\end{defin}

\begin{corol}
Let $\Ac$ be a central essential arrangement of hyperplanes in $\RR^r$, $r\ge 2$. Then $\Ac$ is simplicial if and only if
\begin{equation}\label{combsim0}
r \chi_{\Ac}(-1) + 2 \sum_{H\in\Ac} \chi_{\Ac^H}(-1) = 0.
\end{equation}
\end{corol}
\begin{proof}
By a theorem of Zaslavsky (see \cite[Thm.\ 2.68]{OT}) the number of chambers of an arrangement $\Ac$ is
$(-1)^r\chi_\Ac(-1)$. The number $c_{r-2}$ is the number of walls of $\Ac$, i.e.\ the sum over the numbers of chambers for each restriction $\Ac^H$, $H\in\Ac$.
\end{proof}

Notice that Equation \ref{combsim0} does not depend on simplicial cones or on the fact that $V$ is
a vector space over the real numbers. This motivates the following definition.

\begin{defin}\label{def:combsimp}
Let $K$ be a field and let $\Ac$ be a finite set of hyperplanes in $V=K^r$.
We call $\Ac$ \emph{(combinatorially) simplicial} if $\Ac$ is a central, essential arrangement satisfying
$r \chi_{\Ac}(-1) + 2 \sum_{H\in\Ac} \chi_{\Ac^H}(-1) = 0$.
\end{defin}

Assume now that $V$ is a vector space of dimension $3$.
By Corollary \ref{combsimp_1}, we may rephrase simpliciality for the projective plane in the following way:

\begin{corol}\label{corcombsim}
Let $\Ac$ be a central essential arrangement in $V=K^3$. Let $P$ denote the set of
one-dimensional intersections of hyperplanes of $\Ac$. Then $\Ac$ is simplicial if and only if
\begin{equation}\label{combsim}
3 (|P|-1) = \sum_{v\in P} |\{ H\in\Ac\mid v\subset H\}|.
\end{equation}
\end{corol}

\begin{defin}
Let $K$ be a field and let $\Ac$ be a finite set of hyperplanes in $V=K^3$.
We will call the reducible\footnote{In $K^3$ this is equivalent to the fact that all hyperplanes except one meet in one one-dimensional intersection point.} simplicial arrangements \emph{near pencils}.
For $1<i\in\NN$ let
\[ t_i := |\{ v\in P \mid |\{H\in\Ac \mid v\in H\}| = i\}|. \]
We call $(2^{t_2},3^{t_3},\ldots)$ the \emph{$t$-vector} of $\Ac$.
\end{defin}

\section{Simplicial arrangements over finite fields}

Assume now that $K=\FF_q$ is a finite field and $V=\FF_q^3$.
To simplify the calculations we will henceforth view arrangements in $V$ as sets of projective lines in $\PFq=\PP_2\FF_q$.
For an arrangement $\Ac$ in $\PFq$ we set
\[ n^\Ac_i := |\{ v\in\PFq \mid |\{H\in\Ac \mid v\in H\}| = i\}|. \]

\begin{theor}\label{simpFq}
Let $\Ac$ be an arrangement in $\PFq$. Then $\Ac$ is simplicial if and only if
\[ |\Ac| = 3q-\frac{3n^\Ac_0+2n^\Ac_1}{q+1}. \]
\end{theor}

\begin{proof}
The projective plane $\PFq$ has $q^2+q+1$ points. The set $P$ of intersection points of elements of $\Ac$ has thus
\[ |P| = q^2+q+1-n^\Ac_0-n^\Ac_1 \]
elements. Equation \ref{combsim} is hence equivalent to
\[ 3(q^2+q+1)-3n^\Ac_0-3n^\Ac_1-3 = \sum_{v\in P} |\{ H\in\Ac\mid v\in H\}|. \]
But
\begin{eqnarray*}
\sum_{v\in P} |\{ H\in\Ac\mid v\in H\}|&=&\sum_{v\in \PFq} |\{ H\in\Ac\mid v\in H\}|-n^\Ac_1\\
&=& |\Ac|(q+1)-n^\Ac_1,
\end{eqnarray*}
since each hyperplane contains exactly $q+1$ points in $\PFq$.
Altogether we obtain the claim.
\end{proof}

Since $n^\Ac_0, n^\Ac_1 \ge 0$, we immediately obtain:

\begin{corol}
Let $\Ac$ be a simplicial arrangement in $\PFq$. Then
\[ |\Ac| \le 3q. \]
\end{corol}

\begin{remar}
Notice that this is a sharp bound, see Proposition \ref{Fq3q} or Proposition \ref{Ge13}.
\end{remar}

\begin{remar}
One can check that $n^\Ac_0 = q^2-|\Ac|q+q+f-|\Ac|$
where $2f$ is the number of chambers of $\Ac$. Notice that we thus have
\[ \chi_\Ac(q) = (q-1) n^\Ac_0 \]
where $\chi_\Ac(t)$ is the characteristic polynomial of $\Ac$ (compare \cite[Thm.\ 2.69]{OT}).
\end{remar}

\section{Some prominent examples}

\subsection{A series of arrangements $\Ac$ over $\FF_q$ with $2q\le|\Ac|<3q$}

We first present an infinite series of arrangements whose incidences include the incidences of the arrangements of the family of real simplicial arrangements $\Rc(1)$, see below.
We write $(a,b,c)^\perp$ for the kernel of the matrix $(a\:b\:c)$.

\begin{defin}\label{def:n0series}
Let $\FF_q$ be a finite field,
\begin{eqnarray*}
\Ac&:=&\{(0,1,a)^\perp\mid a\in\FF_q\} \cup \{(1,a,a^2)^\perp\mid a\in\FF_q\},\quad \text{and}\\
\Ec&:=&\{(1,a,0)^\perp\mid a\in\FF_q^\times\}.
\end{eqnarray*}
For a subset $\Bc\subseteq \Ec$, we write $\Dc_\Bc := \Ac\cup\Bc$.
\end{defin}

\begin{theor}\label{n0series}
If $q$ is odd, then the arrangements $\Dc_\Bc$ are simplicial for all $\Bc\subseteq \Ec$.
Moreover,
two arrangements $\Dc_{\Bc_1}$ and $\Dc_{\Bc_2}$
are isomorphic under an element of $\PGL_3(\FF_q)$ if and only if there exists $z\in\FF_q^\times$ such that
\[ \Dc_{\Bc_1} = \begin{pmatrix} 1&0&0\\0&z&0\\0&0&z^2 \end{pmatrix}\Dc_{\Bc_2}. \]
We have an action of the cyclic group $\FF_q^\times$ on
$\{\Dc_\Bc\mid \Bc\subseteq \Ec\}$ which preserves isomorphism classes and obtain
$\frac{1}{n} \sum_{d\mid n} \varphi\left(\frac{n}{d}\right) 2^d$
isomorphism classes of arrangements of the form $\Dc_\Bc$.
\end{theor}
\begin{proof}
Let $\Bc\subseteq \Ec$ and $\Dc:=\Ac\cup\Bc$.
By Thm.\ \ref{simpFq}, it suffices to compute $n_0^\Dc$ and $n_1^\Dc$.
Counting the number of lines through each point in $\PFq$ we obtain $n_0^\Dc=0$ and
\[ n_1^\Dc = \frac{q(q+1)}{2}-|\Bc|\frac{q+1}{2}, \]
which indeed satisfy the simpliciality condition.\\
It is easy to see that $\FF_q^\times$ acts on $\Omega=\{\Dc_\Bc\mid \Bc\subseteq \Ec\}$
via $z\mapsto\diag(1,z,z^2)$.
It remains to show that there are no further $g\in\PGL_3(\FF_q)$ and $\Bc_1,\Bc_2$ with $g \Dc_{\Bc_2}=\Dc_{\Bc_1}$.
For $H\in\Dc_{\Bc_2}$ let
\[ m_H := |\{ H' \in \Dc_{\Bc_2}\backslash \{H\} \mid
H^\perp \subset H'\}|. \]
Then for $a\in\FF_q^\times$,
\begin{eqnarray*}
&&m_{(0,1,0)^\perp}=1,\:\:m_{(0,1,a)^\perp}=3,\:\:m_{(1,0,0)^\perp}=q, \\
&&m_{(1,a,a^2)^\perp}\le 2,\:\: m_{(1,a,0)^\perp}\le 2,
\end{eqnarray*}
and this does not depend on the choice of $\Bc_2$. Thus if $g\Dc_{\Bc_2}=\Dc_{\Bc_1}$, $g\in\PGL_3(\FF_q)$, then $g$ permutes
the set $\{(0,1,a)^\perp\mid a\in\FF_q^\times\}$ and fixes $(1,0,0)^\perp$.

Since there is just one arrangement $\Dc$ with $|\Dc|=|\Ac|$, we may assume $|\Bc_1|=|\Bc_2|>0$.
Let $h\in\GL_3(\FF_q)$ be a representative of $g\in\PGL_3(\FF_q)$, so
\[ h((1,0,0)) = \lambda_1 (1,0,0),\quad
h((0,1,1)) = \lambda_2 (0,1,a),\]
\[h((0,1,2)) = \lambda_3 (0,1,b), \]
for certain $\lambda_1,\lambda_2,\lambda_3\in\FF_q^\times$, $a,b\in\FF_q^\times$, $a \ne b$.
We obtain
\[ h = \begin{pmatrix}
\lambda_1 & 0 & 0 \\
0 & 2\lambda_2-\lambda_3 & \lambda_3-\lambda_2 \\
0 & 2\lambda_2 a-\lambda_3 b & \lambda_3 b-\lambda_2 a
\end{pmatrix}. \]
If $\lambda_2\ne \lambda_3$ then
$h((0,1,\frac{\lambda_3-2\lambda_2}{\lambda_3-\lambda_2})) \in\langle (0,0,1)\rangle$,
thus $\lambda_3-2\lambda_2=0$ since $g$ permutes the set $\{(0,1,a)^\perp\mid a\in\FF_q^\times\}$.
But then $h(H)\notin \Dc_{\Bc_1}$ for every $H\in\Bc_2$ (remember that $\Bc_2\ne\emptyset$) which is a contradiction,
hence $\lambda_2=\lambda_3$ and
\[ h = \begin{pmatrix}
\lambda_1 & 0 & 0 \\
0 & \lambda_2 & 0 \\
0 & \lambda_2 (2a-b) & \lambda_2 (b-a)
\end{pmatrix}. \]
Now $2a=b$ because otherwise $h((0,1,\frac{2a-b}{a-b}))=(0,\lambda_2,0)$ contradicts the fact that
$g$ permutes the set $\{(0,1,a)^\perp\mid a\in\FF_q^\times\}$. Thus $h=\diag(\lambda_1,\lambda_2,\lambda_2a)$.
Now choose any $x\in\FF_q^\times$. We have $h((1,x,x^2))=\lambda(1,y,y^2)$ for certain $\lambda,y\in\FF_q^\times$ because $g\Dc_{\Bc_2}=\Dc_{\Bc_1}$. This implies $\lambda=\lambda_1$ and $\frac{\lambda_2}{\lambda_1}=\frac{y}{x}$ such that
$a=\frac{\lambda_2}{\lambda_1}$, or $h=\lambda_1\diag(1,a,a^2)$.

The claim about the number of isomorphism classes is now a simple application of Polya's Theorem.
\end{proof}

The arrangements $\Ac(2n,1)$ of the infinite series $\Rc(1)$ shown in \cite{p-G-09} may be defined in the following way.
Let $n\in\NN$ and $\zeta:=\exp(2\pi\ii /2n)$ be a primitive $2n$-th root of unity. To make the formulae more readable we will write
\begin{eqnarray*}
c(m)&:=&\cos \frac{2\pi m}{2n} = \frac{\zeta^m+\zeta^{-m}}{2}, \\
s(m)&:=&\sin \frac{2\pi m}{2n} = \frac{\zeta^m-\zeta^{-m}}{2\ii}.
\end{eqnarray*}
Normal vectors of the hyperplanes of $\Ac(2n,1)$ in $\RR^3$ are generators of
\begin{eqnarray*}
&\ker &\begin{pmatrix} 1 & c(2m) & s(2m) \\ 1 & c(2m+2) & s(2m+2) \end{pmatrix},\\
&\ker &\begin{pmatrix} 1 & c(m) & s(m) \\ 1 & 0 & 0 \end{pmatrix},
\end{eqnarray*}
for $m\in\{0,\ldots,n-1\}$. We choose the following generators:
\begin{eqnarray*}
\alpha_m &:=& \big(s(2m)c(2m+2)-c(2m)s(2m+2),\\
& & \:\:s(2m+2)-s(2m),\:\: c(2m)-c(2m+2)\big),\\
\alpha'_m &:=& \big(0,\:\:s(m),\:\:-c(m)\big),
\end{eqnarray*}
for $m = 0\ldots n-1$, so
\[ \Ac(2n,1) = \{\alpha_m^\perp,{\alpha'}_m^\perp \mid m=0,\ldots,n-1\}. \]

\begin{theor}
Let $\FF_q$ be a finite field and $q$ an odd prime. Then the incidence of the arrangement
$\Dc_\emptyset$ is isomorphic to the incidence of the arrangement
$\Ac(2q,1)$ of the infinite series.
\end{theor}

\begin{proof}
It may be checked by determinants of triples of normal vectors that the map
\[ \varphi : \Dc_\emptyset \rightarrow \Ac(2q,1),\quad
(0,1,\overline{a})\mapsto \alpha'_a, \quad
(1,\overline{a},\overline{a}^2)\mapsto \alpha_{(a+1)\md q} \]
for $a\in\{0,\ldots,q-1\}$ induces an isomorphism of incidences.
\end{proof}

We believe that if $q$ is not an odd prime, then the incidence of $\Ac$ is not realizable over $\CC$.

\begin{conje}
Let $q$ be a prime power, $q\notin\PP$.
Then there is no arrangement in $\CC^3$ with the same incidence as the arrangement $\Ac$ with $2q$ hyperplanes defined in Def.\ \ref{def:n0series}.
\end{conje}

However, $\Ac$ is simplicial for arbitrary $q$. The last remaining case is $q$ even:

\begin{propo}
If $q$ is even, then $\Ac$ is simplicial.
\end{propo}
\begin{proof}
As in Thm.\ \ref{n0series} we compute $n_0^\Ac$ and $n_1^\Ac$; they satisfy the simpliciality condition.
\end{proof}

\subsection{A set of arrangements $\Ac$ over $\FF_q$ with $n^\Ac_0=0$ and $|\Ac|=3q$}
We now describe a set of simplicial arrangements over $\FF_q$ with the maximal possible number of hyperplanes.
It turns out that one of them has the same incidence as the imprimitive complex reflection group $G(e,1,3)$ (see Example \ref{Fq3qex} and
Section \ref{simpcomprefl}).

\begin{propo}\label{Fq3q}
Let $\FF_q$ be a finite field and $\Ac$ the arrangement of all lines in $\PFq$.
Choose a line $H\in\Ac$ and $q-1$ points $v_1,\ldots,v_{q-1}$ on $H$.
For each $v_i$, choose $q-1$ further lines $H^i_1,\ldots,H^i_{q-1}$ through $v_i$.
Then
\[ \Bc:=\Ac\backslash \{ H^i_j\mid 1\le i,j \le q-1 \} \]
is a simplicial arrangement with $3q$ lines.
\end{propo}
\begin{proof}
There are $(q^2+q+1)-(q-1)^2=3q$ lines remaining in $\Bc$. Further, for each $v_i$, two lines are left going through $v_i$.
If $u$ is any other point, then we have removed at most $q-1$ lines from the $q+1$ lines through $u$,
since $H^i_j$ can meet $H^k_\ell$ only if $i\ne k$. Thus there are at least two lines of $\Bc$ through $u$.
Hence $n_0=n_1=0$, and $\Bc$ satisfies the equation of Thm.\ \ref{simpFq}.
\end{proof}

\begin{examp}\label{Fq3qex}
Let
\[ \Ac := \{(1,a,0)^\perp,(1,0,a)^\perp,(0,1,a)^\perp \mid a\in\FF_q\} \cup \{(0,0,1)^\perp\}. \]
Then $\Ac$ is a simplicial arrangement obtained as in Proposition \ref{Fq3q} via the line $(1,0,0)^\perp$.
It has the same incidence as the reflection arrangement of the group $G(q-1,1,3)$ by Prop.\ \ref{Ge13H}.
Its $t$-vector is $(2^{3(q-1)},3^{(q-1)^2},(q+1)^3)$ (see Prop.\ \ref{Ge13}).
\end{examp}

\begin{examp}\label{counterex}
Consider the (complex) reflection arrangement of the complex reflection group $G_{25}$,
\begin{eqnarray*}
\Ac &=& \{(0,1,0)^\perp,(0,1,-2\zeta_3 - 1)^\perp,(1,0,-\zeta_3 - 1)^\perp,(1,0,\zeta_3)^\perp,\\
&&(1,-1,\zeta_3)^\perp,(1,1,-\zeta_3 - 1)^\perp,(3,2\zeta_3 + 1,0)^\perp,(1,\zeta_3,1)^\perp,\\
&&(1,\zeta_3,\zeta_3)^\perp,(1,\zeta_3 + 1,1)^\perp,(1,\zeta_3 + 1,-\zeta_3 - 1)^\perp,(1,2\zeta_3 + 1,1)^\perp\}
\end{eqnarray*}
where $\zeta_3$ is a third root of unity. This is a simplicial arrangement with the same incidence as
\begin{eqnarray*}
\Ac' &=& \{
(1,1,0)^\perp,
(1,1,\omega)^\perp,
(0,0,1)^\perp,
(1,0,\omega)^\perp,\\
&&(0,1,\omega^2)^\perp,
(1,\omega^2,\omega^2)^\perp,
(0,1,0)^\perp,
(1,1,1)^\perp,\\
&&(1,\omega^2,0)^\perp,
(1,\omega,0)^\perp,
(0,1,\omega)^\perp,
(1,\omega^2,\omega)^\perp
\}
\end{eqnarray*}
in $\FF_4^3$, where $\omega$ is a primitive element. Notice that $\Ac'$ has $3q=12$ hyperplanes but is not one of the arrangements constructed in Prop.\ \ref{Fq3q}.
\end{examp}

\subsection{Further arrangements with the incidence of $G(e,1,3)$}

More generally, the incidences of the reflection arrangements $G(e,1,3)$ may be obtained in the following way.

\begin{propo}\label{Ge13H}
Let $q$ be a prime power and $H\le \FF_q^\times$ a subgroup. Then
\[ \Ac = \{(1,a,0)^\perp,(1,0,a)^\perp,(0,1,a)^\perp \mid a\in H\cup\{0\}\} \cup \{(0,0,1)^\perp\} \]
is a simplicial arrangement which has the same incidence as the reflection arrangement of the group $G(|H|,1,3)$.
\end{propo}

\begin{proof}
It is easy to see that the reflecting hyperplanes of $G(e,d,3)$ are
\begin{eqnarray*}
\Bc &=& \{(1,-\zeta^a,0)^\perp,(1,0,-\zeta^a)^\perp,(0,1,-\zeta^a)^\perp \mid a\in\{1,\ldots,e\}\} \cup\\
&& \{(1,0,0)^\perp,(0,1,0)^\perp,(0,0,1)^\perp\},
\end{eqnarray*}
where $\zeta$ is a primitive $e$-th root of unity.
Assume now that $e=|H|$ and that $H=\langle\omega\rangle$. Then the map
\[ \rho : H\cup \{0\}\rightarrow \CC, \quad 0\mapsto 0,\:\: \omega^i\mapsto \zeta^i \]
induces a bijection $\Ac\rightarrow \Bc$ preserving the incidence structure
(check that for each triple $(u,v,w)$ of elements in $\Ac$, $\det(u,v,w)=0$ iff $\det (\rho(u),\rho(v),\rho(w))=0$).
For the simpliciality of $\Ac$ we may now use Prop.\ \ref{Ge13}.
\end{proof}

\section{Simplicial arrangements by symmetries}

In this section we compute examples of simplicial arrangements $\Ac$ in $\FF_q^3$ under the assumption that the subgroup of $\PGL_3$ under which $\Ac$ is invariant is non trivial.
We start with a short counting argument to make clear that it will be difficult to enumerate arrangements without using symmetries.
Consider the group $\PGL_3(\FF_q)$ as a permutation group acting on $\PFq$, i.e.\ label the elements (or equivalently the lines) of $\PFq=\{H_1,\ldots,H_{q^2+q+1}\}$ and $\PGL_3(\FF_q) \le \Sym_{q^2+q+1}$.

\begin{propo}\label{pg2uptopgl}
Let $C$ be the set of conjugacy classes of $\PGL_3(\FF_q)$. Then the coefficient before $t^k$ in the polynomial
\[ F(t) = \frac{1}{q^8 - q^6 - q^5 + q^3}\sum_{\overline{\sigma}\in C} |\overline{\sigma}| \prod_{\tau \text{ cycle of }\sigma} (1+t^{|\tau|}) \]
counts the number of arrangements with $k$ lines in $\PFq$ up to projectivities, where $\overline{\sigma}$ denotes the class of $\sigma$ and $|\tau|$ denotes the size of the cycle $\tau$.
\end{propo}
\begin{proof}
We apply Polya's Theorem to the action of $\PGL_3(\FF_q)$ on the set of all subsets of $\PFq$ and get
\[ F(t) = \frac{1}{|\PGL_3(\FF_q)|}\sum_{\sigma\in \PGL_3(\FF_q)} \prod_{\tau \text{ cycle of }\sigma} (1+t^{|\tau|}).\]
Since the cycle structure is the same up to conjugacy, we obtain the above formula.
\end{proof}
\begin{examp}
Some small numbers $n$ of arrangements of lines in $\PFq$ up to projectivities are listed in Fig.\ \ref{uptoproj}.
\begin{figure}
{\tiny
\begin{tabular}{c | c c c c c c c}
$q$ & 2 & 3 & 4 & 5 & 7 & 8 & 9 \\ \hline
$n$ & 10 & 30 & 160 & 7152 & 25598921348 & 573005431135008 & 58315058241829513832 \\
$n_{\le 3q}$ & 9 & 25 & 116 & 3576 & 803236855 & 1325456501156 & 3853555217682705
\end{tabular}
}
\caption{Arrangements of lines in $\PFq$ up to $\PGL_3(\FF_q)$.\label{uptoproj}}
\end{figure}
The numbers denoted $n_{\le 3q}$ are the arrangements of lines in $\PFq$ with at most $3q$ lines up to projectivities.
Thus a brute force enumeration of arrangements to find all the simplicial ones is (if at all) feasible only up to $q=7$, provided that we have a good algorithm to compute canonical representatives.
\end{examp}

Figure \ref{simpinctable} shows statistics for the number of simplicial arrangements up to collineations in $\PFq$ for small $q$.
Notice that some of these arrangements are near pencils.
\begin{figure}
\begin{tiny}
\begin{tabular}{l l}
$q=3$: &
\begin{tabular}{l|rrrrrrr}
Hyperplanes & 3 & 4 & 5 & 6 & 7 & 8 & 9 \\
\hline
Incidences & 1 & 1 & 1 & 1 & 1 & 1 & 1
\end{tabular} \\
$q=4$: &
\begin{tabular}{l|rrrrrrrrrrr}
Hyperplanes & 3 & 4 & 5 & 6 & 7 & 8 & 9 & 10 & 11 & 12 \\
\hline
Incidences & 1 & 1 & 1 & 2 & 0 & 1 & 0 & 3 & 2 & 3
\end{tabular}\\
$q=5$: &
\begin{tabular}{l|rrrrrrrrrrrrrr}
Hyperplanes & 3 & 4 & 5 & 6 & 7 & 8 & 9 & 10 & 11 & 12 & 13 & 14 & 15 \\
\hline
Incidences & 1 & 1 & 1 & 2 & 2 & 1 & 1 & 3 & 5 & 39 & 146 & 77 & 6
\end{tabular}
\end{tabular}
\end{tiny}
\caption{Incidences of simplicial arrangements over $\FF_q$.\label{simpinctable}}
\end{figure}

But now assume that $\Ac$ is an arrangement of lines in $\PFq$ stable under the action of some non-trivial subgroup $G\le\PGL_3(\FF_q)$, i.e.\ $\sigma(\Ac)=\Ac$ for all $\sigma\in G$. Then $\Ac$ is a union of hyperplanes orthogonal to elements of orbits of $G$ on $\PFq$.

In the following, we choose a nice group $G$ which occurs as symmetry group of a real simplicial arrangement, and consider the simplicial arrangements in $\FF_q^3$ which are unions of orbits.

So let $G$ be an abstract group. We need to understand the possible embeddings of $G$ into $\PGL_3(\CC)$ first. A homomorphism $G\rightarrow \PGL_3(\CC)$ is called a projective ($3$-dimensional) representation of $G$. By a theorem of Schur, there exists a central extension $1\rightarrow A\rightarrow \Gamma\rightarrow G\rightarrow 1$ such that $\Gamma$ has the lifting property for $G$, i.\ e.\ every projective representation is obtained as an ordinary representation of $\Gamma$ which is a scalar representation when restricted to $A$.
There is an explicit construction of such an extension such that $A$ is dual to the Schur multiplier $M(G):=H^2(G,\CC^\times)$.
Multipliers are easy to compute for the few small groups that we will consider as symmetry groups below (see \cite{MR1200015}).

We have no technique to perform exhaustive searches for arrangements with given symmetry group. Any experiments in this direction need bounds for the size of the field and for the number of orbits of $G$ on $\PFq$. We thus content ourselves with some noticeable examples and leave a systematic search as an open problem.

We find far too many simplicial arrangements over $\FF_q$ with small symmetry groups to even reproduce their invariants here. Of course, this number decreases with increasing number of symmetries. Nevertheless, we will only present those incidences which are relevant for characteristic zero.

\subsection{Cyclic groups of order $q-1$}

Looking for cyclic groups is comparatively easy in explicit computations. Since we are interested in cyclic subgroups of $\PGL_3(\FF_q)$ up to conjugacy, determining the cyclic subgroups more or less amounts to determining the conjugacy classes of $\PGL_3(\FF_q)$ (see for example \cite{MR0041851}).
A series of experiments shows that the only ``simplicial'' incidences $I$ that we obtain for the symmetry group $\langle\sigma\rangle$, $\sigma\in\PGL_3(\FF_q)$ ($q$ large) of order $q-1$ such that $I$ is realizable over $\CC$ are the incidences of $\Ac(2n,1)$, $\Ac(4n+1,1)$, and of the reflection groups $G(e,1,3)$.
Of course, this does not prove anything; a complete classification of these incidences would be a remarkable step towards a classification of simplicial arrangements.

\subsection{Other small symmetry groups}\label{ossg}

For subgroups $G\le \PGL_3(\FF_q)$ isomorphic to a cyclic group, $S_3$, $S_4$, or to some small dihedral groups
we have a complete list of simplicial arrangements over $\FF_q$, $3\le q\le 13$ with at least $2q$ hyperplanes such that the number of orbits of sizes less or equal to $3q$ under $G$ is at most $20$ (Fig.\ \ref{simC}). We list only those which are realizable over $\CC$ (see \cite{p-C10b} for an algorithm computing a realization), which have an incidence which is not the incidence of a known simplicial arrangement, and do not come from a complex reflection group.
Notice that we list each \emph{incidence} only once, i.e.\ if an incidence is realizable over $q=7$ and $q=9$ then it will appear only once in the table.

The column denoted `$K$' displays a (possibly minimal) field extension of $\QQ$ over which the incidence is realizable; $\zeta_n$ denotes a primitive $n$-th root of unity in $\CC$. The automorphism groups listed in the last columns are the groups of bijections from $\Ac$ to $\Ac$ preserving the incidence relation. These will not necessarily coincide with the chosen symmetry group $G$.

\begin{figure}
\begin{longtable}{|r r l r l r|}
\hline
$|\Ac|$ & $|P|$ & $K$ & $q$ & $t$-vector & $|\Aut(I)|$ \\
\hline
12 & 22 & $\QQ(\zeta_{4})$ & 5 & $2^{7},3^{13},5^{2}$ & 16 \\
14 & 29 & $\QQ(\zeta_{3})$ & 7 & $2^{13},3^{6},4^{10}$ & 12 \\
15 & 33 & $\QQ(\zeta_{3})$ & 7 & $2^{14},3^{9},4^{9},5^{1}$ & 36 \\
15 & 33 & $\QQ(\zeta_{3})$ & 7 & $2^{15},3^{6},4^{12}$ & 6 \\
15 & 33 & $\QQ(\zeta_{3})$ & 7 & $2^{11},3^{18},5^{4}$ & 16 \\
15 & 34 & $\QQ(\zeta_{4})$ & 9 & $2^{12},3^{13},4^{9}$ & 6 \\
15 & 35 & $\QQ(\zeta_{3})$ & 7 & $2^{9},3^{20},4^{6}$ & 16 \\
16 & 37 & $\QQ(\zeta_{3})$ & 7 & $2^{18},3^{4},4^{15}$ & 8 \\
17 & 41 & $\QQ(\zeta_{3})$ & 7 & $2^{20},3^{6},4^{13},5^{2}$ & 8 \\
17 & 45 & $\QQ(\zeta_{7})$ & 7 & $2^{10},3^{28},4^{7}$ & 8 \\
18 & 45 & $\QQ(\zeta_{3})$ & 7 & $2^{25},4^{18},5^{2}$ & 36 \\
18 & 45 & $\QQ(\zeta_{3})$ & 7 & $2^{24},3^{3},4^{15},5^{3}$ & 6 \\
18 & 46 & $\QQ(\zeta_{4})$ & 9 & $2^{18},3^{19},4^{3},5^{6}$ & 12 \\
18 & 49 & $\QQ(\zeta_{7})$ & 7 & $2^{15},3^{22},4^{12}$ & 6 \\
18 & 49 & $\QQ(\zeta_{7})$ & 7 & $2^{15},3^{22},4^{12}$ & 24 \\
18 & 49 & $\QQ(\zeta_{4})$ & 9 & $2^{15},3^{22},4^{12}$ & 6 \\
19 & 49 & $\QQ(\zeta_{3})$ & 7 & $2^{28},3^{1},4^{15},5^{5}$ & 12 \\
19 & 51 & $\QQ(\zeta_{4})$ & 9 & $2^{21},3^{18},4^{6},5^{6}$ & 16 \\
19 & 54 & $\QQ(\zeta_{7})$ & 9 & $2^{15},3^{30},4^{6},5^{3}$ & 12 \\
19 & 55 & $\QQ(\zeta_{4})$ & 9 & $2^{13},3^{34},4^{6},5^{2}$ & 16 \\
20 & 53 & $\QQ(\zeta_{3})$ & 7 & $2^{32},4^{13},5^{8}$ & 48 \\
21 & 69 & $\QQ(\zeta_{3})$ & 13 & $2^{12},3^{48},4^{9}$ & 48 \\
22 & 67 & $\QQ(\zeta_{8})$ & 9 & $2^{16},3^{48},6^{1},8^{2}$ & 32 \\
22 & 67 & $\QQ(\zeta_{8})$ & 9 & $2^{16},3^{48},6^{1},8^{2}$ & 128 \\
22 & 67 & $\QQ(\zeta_{8})$ & 9 & $2^{16},3^{48},6^{1},8^{2}$ & 64 \\
22 & 74 & $\QQ(\zeta_{12})$ & 9 & $2^{13},3^{56},5^{5}$ & 32 \\
25 & 85 & $\QQ(\zeta_{7})$ & 11 & $2^{36},3^{28},4^{15},6^{6}$ & 24 \\
30 & 129 & $\QQ(\zeta_{3})$ & 13 & $2^{36},3^{78},5^{12},6^{3}$ & 144 \\
31 & 127 & $\QQ(\zeta_{4})$ & 13 & $2^{48},3^{64},6^{15}$ & 192 \\
\hline
\end{longtable}
\caption{Simplicial arrangements over $\CC$ found using symmetries \label{simC}}
\end{figure}

\subsection{Further examples}

\begin{examp}
The strangest example that we obtain in this manner is the following.
Let $H$ be the group of $3\times 3$ monomial matrices over $\FF_q$, $q=17$ with entries $\pm 1$. Thus $H$ is isomorphic to the reflection group of type $B_3$. Then
\[ \Ac := H (0,0,1)^\perp \cup H (0,1,1)^\perp \cup H (1,1,3)^\perp \cup H (1,1,7)^\perp \]
is a simplicial arrangement with $33$ hyperplanes, the orbits have sizes $3,6,12,12$ respectively.
But the most curious property of this arrangement is that there exists a realization of its incidence over the minimal number field $K$ with defining polynomial $x^4 + 186 x^2 - 1400 x + 17597$. This field is not Galois over $\QQ$, and this happens extremely rarely for simplicial arrangements: Over the reals, there are only two known simplicial arrangements (denoted $\Ac(15,5)$ and $\Ac(21,7)$ in \cite{p-G-09}) having as minimal field a non-abelian field extension of $\QQ$ (see \cite{p-C10b} for an algorithm to compute these fields).
\end{examp}

\begin{examp}
Choosing for $H$ the group of type $B_3$ over $\FF_q$, $q=11$, among others we find the incidence of the recently discovered arrangement denoted $\Ac(25,8)$ in \cite{p-C12}:
\[ \Ac := H (0,0,1)^\perp \cup H (1,1,1)^\perp \cup H (0,1,1)^\perp \cup H (0,1,3)^\perp \]
is simplicial and has orbit sizes $3,4,6,12$. Here the minimal field for a realization of the incidence over $\CC$ is $\QQ(\sqrt{5})$.
\end{examp}

\section{Simplicial reflection arrangements}\label{simpcomprefl}

We now determine which irreducible reflection arrangements are simplicial.
All real reflection groups yield simplicial arrangements in the usual sense, so we may concentrate here on the remaining complex reflection groups.

All involved arrangements in this section are \emph{free}, which means that their modules of derivations are free. The degrees of homogeneous generators of these modules are called the \emph{exponents} of the arrangements.
Terao's Factorization Theorem (see \cite[Thm.\ 4.137]{OT}) allows us to compute all required characteristic polynomials easily: If $\Ac$ is free with exponents $\lmultiset e_1,\ldots,e_r\rmultiset$, then $\chi_\Ac(t)=\prod_{i=1}^r (t-e_i)$.

\begin{propo}\label{Ge13}
The reflection arrangement $\Ac_r$ of the complex reflection group $G(e,d,r)$, $e\ne d$, $d|e$, $r>1$ is simplicial.
\end{propo}
\begin{proof}
By \cite[Prop.\ 6.77]{OT}, $\Ac_r$ is free with exponents
\[ \lmultiset 1,e+1,2e+1,\ldots,(r-1)e+1 \rmultiset. \]
Thus $\chi_{\Ac_r}(-1)=\prod_{k=0}^{r-1} (-1-(1+ke))$ and the number of hyperplanes in $\Ac_r$ is $r+\binom{r}{2}e$.
Again by \cite[Prop.\ 6.77]{OT}, the exponents of $\Ac_r^H$ for any $H\in\Ac_r$ are
\[ \lmultiset 1,e+1,2e+1,\ldots,(r-2)e+1 \rmultiset. \]
Thus $\chi_{\Ac_r^H}(-1)=\prod_{k=0}^{r-2} (-1-(1+ke))$. It is now easy to check
$r \chi_{\Ac_r}(-1) + 2 \sum_{H\in\Ac_r} \chi_{\Ac_r^H}(-1) = 0$; hence $\Ac_r$ is simplicial. 
\end{proof}

\begin{propo}\label{Gee3}
The reflection arrangement $\Ac^0_r$ of the complex reflection group $G(e,e,r)$, $e,r>1$ is simplicial if and only if $e=2$ or $r=2$.
\end{propo}
\begin{proof}
By \cite[Prop.\ 6.85]{OT}, $\Ac^0_r$ is free with exponents
\[ \lmultiset 1,e+1,\ldots,(r-2)e+1,(r-1)(e-1) \rmultiset, \]
so the number of hyperplanes in $\Ac^0_r$ is $\binom{r}{2}e$.
By \cite[Prop.\ 6.84 and Prop.\ 6.85]{OT}, the exponents of ${\Ac^0_r}^H$ for any $H\in\Ac^0_r$ are
\[ \lmultiset 1,e+1,\ldots,(r-3)e+1,(r-2)e-r+3 \rmultiset. \]
One computes that $r \chi_{\Ac^0_r}(-1) + 2 \sum_{H\in\Ac^0_r} \chi_{{\Ac^0_r}^H}(-1) = 0$ if and only if
$r(r-2)(e-2)=0$. Thus $\Ac^0_r$ is simplicial if and only if $e=2$ or $r=2$.
\end{proof}

\begin{propo}
The exceptional complex reflection groups of dimension greater than $2$ are the Shephard-Todd groups $G_{23},\ldots,G_{37}$.
Among them, only the reflection arrangements of $G_{24},G_{27}$, $G_{29}$, $G_{33}$, and $G_{34}$ are not simplicial.
\end{propo}
\begin{proof}
The groups $G_{23}$, $G_{30}$, $G_{35}$, $G_{36}$, and $G_{37}$ are Coxeter groups and thus define simplicial arrangements.
A direct computation partly using the tables in \cite[Appendix C]{OT} handles the remaining cases.
\end{proof}

It is now easy to deduce the result mentioned in the introduction with \cite[Cor.\ 5.16]{p-BC10} and \cite[Thm.\ 1.1]{p-HR12}:
\begin{corol}
The reflection arrangement of a finite irreducible complex reflection group not isomorphic to $G_{31}$ is combinatorially simplicial if and only if it is inductively free. The reflection arrangement of $G_{31}$ is simplicial but not inductively free.
\end{corol}

\appendix
\section{Some simplicial arrangements over $\CC$}

We reproduce here the newly found simplicial arrangements over $\CC$ in the most compact notation, namely as sets of normal vectors in the original finite field. The ordering is the same as in \ref{ossg}. The symbol $\omega$ stands for a primitive element of the corresponding field.
As in \ref{ossg}, we omit those incidences which come from real simplicial arrangements, from finite complex reflection groups, or which are near pencil.

{\tiny
\noindent
$q=5$,\:\: $\{(0,1,1)$, $(1,0,0)$, $(1,0,1)$, $(1,0,2)$, $(1,1,0)$, $(1,2,1)$, $(1,3,2)$, $(1,4,0)$, $(1,4,1)$, $(1,4,2)$, $(1,4,3)$, $(1,4,4)\}$, \\
$q=7$,\:\: $\{(0,1,0)$, $(0,1,2)$, $(0,1,3)$, $(0,1,4)$, $(1,0,0)$, $(1,0,1)$, $(1,1,3)$, $(1,2,0)$, $(1,2,2)$, $(1,3,0)$, $(1,4,1)$, $(1,5,1)$, $(1,6,3)$, $(1,6,5)\}$, \\
$q=7$,\:\: $\{(0,1,0)$, $(0,1,2)$, $(0,1,3)$, $(0,1,4)$, $(1,0,0)$, $(1,0,1)$, $(1,0,5)$, $(1,1,3)$, $(1,2,0)$, $(1,2,2)$, $(1,3,0)$, $(1,4,1)$, $(1,5,1)$, $(1,6,3)$, $(1,6,5)\}$, \\
$q=7$,\:\: $\{(0,0,1)$, $(1,0,0)$, $(1,0,5)$, $(1,0,6)$, $(1,1,0)$, $(1,1,2)$, $(1,1,6)$, $(1,2,0)$, $(1,2,2)$, $(1,2,6)$, $(1,3,0)$, $(1,3,5)$, $(1,5,1)$, $(1,5,3)$, $(1,5,4)\}$, \\
$q=7$,\:\: $\{(0,1,2)$, $(1,0,0)$, $(1,0,2)$, $(1,1,4)$, $(1,1,5)$, $(1,1,6)$, $(1,2,3)$, $(1,3,0)$, $(1,3,2)$, $(1,3,4)$, $(1,4,0)$, $(1,4,1)$, $(1,4,4)$, $(1,5,6)$, $(1,6,6)\}$, \\
$q=9$,\:\: $\{(0,0,1)$, $(1,0,0)$, $(1,0,\omega^2)$, $(1,0,2)$, $(1,1,2)$, $(1,\omega^3,2)$, $(1,\omega^3,\omega^6)$, $(1,2,\omega^3)$, $(1,2,\omega^5)$, $(1,\omega^5,0)$, $(1,\omega^5,1)$, $(1,\omega^5,2)$, $(1,\omega^6,0)$, $(1,\omega^6,\omega^5)$, $(1,\omega^6,\omega^7)\}$, \\
$q=7$,\:\: $\{(0,1,2)$, $(1,0,0)$, $(1,1,4)$, $(1,1,5)$, $(1,1,6)$, $(1,2,2)$, $(1,2,3)$, $(1,3,0)$, $(1,3,4)$, $(1,4,0)$, $(1,4,1)$, $(1,4,4)$, $(1,5,6)$, $(1,6,2)$, $(1,6,6)\}$, \\
$q=7$,\:\: $\{(0,0,1)$, $(0,1,2)$, $(1,0,1)$, $(1,2,0)$, $(1,2,2)$, $(1,2,6)$, $(1,3,0)$, $(1,4,0)$, $(1,4,1)$, $(1,4,2)$, $(1,5,0)$, $(1,5,2)$, $(1,5,3)$, $(1,6,2)$, $(1,6,4)$, $(1,6,5)\}$, \\
$q=7$,\:\: $\{(0,0,1)$, $(0,1,2)$, $(1,0,0)$, $(1,0,1)$, $(1,2,0)$, $(1,2,2)$, $(1,2,6)$, $(1,3,0)$, $(1,4,0)$, $(1,4,1)$, $(1,4,2)$, $(1,5,0)$, $(1,5,2)$, $(1,5,3)$, $(1,6,2)$, $(1,6,4)$, $(1,6,5)\}$, \\
$q=7$,\:\: $\{(0,0,1)$, $(0,1,2)$, $(1,0,0)$, $(1,0,1)$, $(1,0,2)$, $(1,1,4)$, $(1,1,6)$, $(1,2,2)$, $(1,2,6)$, $(1,3,0)$, $(1,3,2)$, $(1,3,4)$, $(1,4,0)$, $(1,4,1)$, $(1,5,6)$, $(1,6,2)$, $(1,6,4)\}$, \\
$q=7$,\:\: $\{(0,0,1)$, $(0,1,0)$, $(0,1,3)$, $(0,1,4)$, $(1,0,0)$, $(1,0,1)$, $(1,0,2)$, $(1,1,0)$, $(1,1,1)$, $(1,1,3)$, $(1,2,0)$, $(1,2,4)$, $(1,2,5)$, $(1,3,0)$, $(1,4,1)$, $(1,4,5)$, $(1,4,6)$, $(1,6,5)\}$, \\
$q=7$,\:\: $\{(0,0,1)$, $(0,1,0)$, $(1,0,0)$, $(1,0,2)$, $(1,0,5)$, $(1,0,6)$, $(1,1,0)$, $(1,1,2)$, $(1,1,6)$, $(1,2,0)$, $(1,2,2)$, $(1,2,6)$, $(1,3,0)$, $(1,3,5)$, $(1,4,5)$, $(1,5,1)$, $(1,5,3)$, $(1,5,4)\}$, \\
$q=9$,\:\: $\{(0,0,1)$, $(1,0,0)$, $(1,0,\omega)$, $(1,0,\omega^2)$, $(1,0,2)$, $(1,1,2)$, $(1,1,\omega^7)$, $(1,\omega^3,2)$, $(1,\omega^3,\omega^6)$, $(1,2,\omega^3)$, $(1,2,\omega^5)$, $(1,\omega^5,0)$, $(1,\omega^5,1)$, $(1,\omega^5,2)$, $(1,\omega^6,0)$, $(1,\omega^6,2)$, $(1,\omega^6,\omega^5)$, $(1,\omega^6,\omega^7)\}$, \\
$q=7$,\:\: $\{(0,1,0)$, $(0,1,1)$, $(0,1,3)$, $(1,0,3)$, $(1,0,5)$, $(1,0,6)$, $(1,1,1)$, $(1,1,5)$, $(1,3,3)$, $(1,3,6)$, $(1,4,5)$, $(1,4,6)$, $(1,5,0)$, $(1,5,2)$, $(1,5,3)$, $(1,6,1)$, $(1,6,2)$, $(1,6,4)\}$, \\
$q=7$,\:\: $\{(0,0,1)$, $(0,1,1)$, $(0,1,2)$, $(1,0,1)$, $(1,0,2)$, $(1,1,1)$, $(1,1,4)$, $(1,1,6)$, $(1,2,6)$, $(1,3,0)$, $(1,3,2)$, $(1,3,4)$, $(1,4,0)$, $(1,4,1)$, $(1,4,6)$, $(1,5,4)$, $(1,5,6)$, $(1,6,4)\}$, \\
$q=9$,\:\: $\{(0,0,1)$, $(1,0,0)$, $(1,0,\omega^2)$, $(1,0,2)$, $(1,1,2)$, $(1,1,\omega^6)$, $(1,\omega^3,2)$, $(1,\omega^3,\omega^6)$, $(1,\omega^3,\omega^7)$, $(1,2,\omega^3)$, $(1,2,\omega^5)$, $(1,2,\omega^6)$, $(1,\omega^5,0)$, $(1,\omega^5,1)$, $(1,\omega^5,2)$, $(1,\omega^6,0)$, $(1,\omega^6,\omega^5)$, $(1,\omega^6,\omega^7)\}$, \\
$q=7$,\:\: $\{(0,0,1)$, $(0,1,0)$, $(0,1,3)$, $(0,1,4)$, $(1,0,0)$, $(1,0,1)$, $(1,0,2)$, $(1,0,5)$, $(1,1,0)$, $(1,1,1)$, $(1,1,3)$, $(1,2,0)$, $(1,2,4)$, $(1,2,5)$, $(1,3,0)$, $(1,4,1)$, $(1,4,5)$, $(1,4,6)$, $(1,6,5)\}$, \\
$q=9$,\:\: $\{(0,1,1)$, $(1,0,\omega)$, $(1,1,0)$, $(1,1,\omega^2)$, $(1,1,\omega^7)$, $(1,\omega,1)$, $(1,\omega,\omega^7)$, $(1,\omega^2,\omega^6)$, $(1,\omega^3,\omega)$, $(1,\omega^3,\omega^3)$, $(1,2,\omega^3)$, $(1,2,\omega^7)$, $(1,\omega^5,1)$, $(1,\omega^6,\omega^7)$, $(1,\omega^7,0)$, $(1,\omega^7,2)$, $(1,\omega^7,\omega^5)$, $(1,\omega^7,\omega^6)$, $(1,\omega^7,\omega^7)\}$, \\
$q=9$,\:\: $\{(0,1,\omega^5)$, $(1,0,0)$, $(1,0,\omega)$, $(1,0,\omega^2)$, $(1,0,2)$, $(1,1,0)$, $(1,1,1)$, $(1,1,2)$, $(1,1,\omega^5)$, $(1,1,\omega^7)$, $(1,\omega,\omega^7)$, $(1,\omega^2,\omega^5)$, $(1,\omega^5,1)$, $(1,\omega^5,2)$, $(1,\omega^6,0)$, $(1,\omega^6,2)$, $(1,\omega^6,\omega^5)$, $(1,\omega^6,\omega^7)$, $(1,\omega^7,\omega^3)\}$, \\
$q=9$,\:\: $\{(0,1,1)$, $(1,0,\omega)$, $(1,1,0)$, $(1,1,\omega^2)$, $(1,1,\omega^7)$, $(1,\omega,1)$, $(1,\omega,\omega^7)$, $(1,\omega^3,\omega)$, $(1,\omega^3,\omega^3)$, $(1,\omega^3,\omega^6)$, $(1,2,\omega^3)$, $(1,2,\omega^6)$, $(1,2,\omega^7)$, $(1,\omega^5,1)$, $(1,\omega^6,\omega^7)$, $(1,\omega^7,0)$, $(1,\omega^7,2)$, $(1,\omega^7,\omega^5)$, $(1,\omega^7,\omega^7)\}$, \\
$q=7$,\:\: $\{(0,0,1)$, $(0,1,2)$, $(1,0,1)$, $(1,0,4)$, $(1,0,6)$, $(1,1,3)$, $(1,2,0)$, $(1,2,2)$, $(1,2,6)$, $(1,3,0)$, $(1,4,0)$, $(1,4,1)$, $(1,4,2)$, $(1,4,5)$, $(1,5,0)$, $(1,5,2)$, $(1,5,3)$, $(1,6,2)$, $(1,6,4)$, $(1,6,5)\}$, \\
$q=13$,\:\: $\{(0,1,10)$, $(1,0,0)$, $(1,0,3)$, $(1,0,4)$, $(1,1,2)$, $(1,2,8)$, $(1,3,0)$, $(1,4,2)$, $(1,4,7)$, $(1,4,9)$, $(1,5,8)$, $(1,6,0)$, $(1,8,1)$, $(1,9,2)$, $(1,9,5)$, $(1,9,9)$, $(1,9,11)$, $(1,10,1)$, $(1,10,5)$, $(1,10,9)$, $(1,11,1)\}$, \\
$q=9$,\:\: $\{(0,0,1)$, $(0,1,1)$, $(0,1,\omega)$, $(0,1,\omega^3)$, $(0,1,2)$, $(0,1,\omega^6)$, $(1,1,0)$, $(1,1,\omega^7)$, $(1,\omega,0)$, $(1,\omega,1)$, $(1,\omega^2,0)$, $(1,\omega^2,\omega)$, $(1,\omega^3,0)$, $(1,\omega^3,\omega^2)$, $(1,2,0)$, $(1,2,\omega^3)$, $(1,\omega^5,0)$, $(1,\omega^5,2)$, $(1,\omega^6,0)$, $(1,\omega^6,\omega^5)$, $(1,\omega^7,0)$, $(1,\omega^7,\omega^6)\}$, \\
$q=9$,\:\: $\{(0,1,1)$, $(0,1,\omega)$, $(0,1,\omega^3)$, $(0,1,2)$, $(0,1,\omega^5)$, $(0,1,\omega^7)$, $(1,1,\omega^2)$, $(1,1,\omega^6)$, $(1,\omega,\omega^3)$, $(1,\omega,\omega^7)$, $(1,\omega^2,1)$, $(1,\omega^2,2)$, $(1,\omega^3,\omega)$, $(1,\omega^3,\omega^5)$, $(1,2,\omega^2)$, $(1,2,\omega^6)$, $(1,\omega^5,\omega^3)$, $(1,\omega^5,\omega^7)$, $(1,\omega^6,1)$, $(1,\omega^6,2)$, $(1,\omega^7,\omega)$, $(1,\omega^7,\omega^5)\}$, \\
$q=9$,\:\: $\{(0,0,1)$, $(0,1,0)$, $(0,1,1)$, $(0,1,\omega)$, $(0,1,2)$, $(0,1,\omega^7)$, $(1,1,\omega^2)$, $(1,1,\omega^6)$, $(1,\omega,\omega^3)$, $(1,\omega,\omega^7)$, $(1,\omega^2,1)$, $(1,\omega^2,2)$, $(1,\omega^3,\omega)$, $(1,\omega^3,\omega^5)$, $(1,2,\omega^2)$, $(1,2,\omega^6)$, $(1,\omega^5,\omega^3)$, $(1,\omega^5,\omega^7)$, $(1,\omega^6,1)$, $(1,\omega^6,2)$, $(1,\omega^7,\omega)$, $(1,\omega^7,\omega^5)\}$, \\
$q=9$,\:\: $\{(0,0,1)$, $(0,1,0)$, $(0,1,\omega^2)$, $(0,1,2)$, $(0,1,\omega^6)$, $(1,0,0)$, $(1,1,\omega^3)$, $(1,1,\omega^7)$, $(1,\omega,\omega^2)$, $(1,\omega,\omega^6)$, $(1,\omega^2,\omega)$, $(1,\omega^2,\omega^5)$, $(1,\omega^3,1)$, $(1,\omega^3,2)$, $(1,2,\omega^3)$, $(1,2,\omega^7)$, $(1,\omega^5,\omega^2)$, $(1,\omega^5,\omega^6)$, $(1,\omega^6,\omega)$, $(1,\omega^6,\omega^5)$, $(1,\omega^7,1)$, $(1,\omega^7,2)\}$, \\
$q=11$,\:\: $\{(0,1,3)$, $(0,1,7)$, $(1,0,4)$, $(1,0,6)$, $(1,1,4)$, $(1,1,6)$, $(1,1,9)$, $(1,1,10)$, $(1,2,4)$, $(1,2,6)$, $(1,2,8)$, $(1,3,4)$, $(1,3,10)$, $(1,4,0)$, $(1,4,4)$, $(1,4,10)$, $(1,5,6)$, $(1,6,2)$, $(1,6,3)$, $(1,6,5)$, $(1,7,10)$, $(1,8,4)$, $(1,9,0)$, $(1,9,1)$, $(1,10,7)\}$, \\
$q=13$,\:\: $\{(0,1,9)$, $(1,0,4)$, $(1,1,10)$, $(1,1,11)$, $(1,1,12)$, $(1,2,7)$, $(1,2,8)$, $(1,2,11)$, $(1,3,7)$, $(1,4,7)$, $(1,5,5)$, $(1,6,0)$, $(1,6,6)$, $(1,6,8)$, $(1,7,6)$, $(1,7,8)$, $(1,7,10)$, $(1,8,5)$, $(1,9,4)$, $(1,9,9)$, $(1,9,12)$, $(1,10,5)$, $(1,10,11)$, $(1,10,12)$, $(1,11,0)$, $(1,11,3)$, $(1,11,10)$, $(1,12,0)$, $(1,12,3)$, $(1,12,6)\}$, \\
$q=13$,\:\: $\{(0,1,7)$, $(1,0,4)$, $(1,0,6)$, $(1,0,9)$, $(1,1,2)$, $(1,2,5)$, $(1,2,8)$, $(1,3,8)$, $(1,4,3)$, $(1,4,7)$, $(1,4,10)$, $(1,5,3)$, $(1,6,0)$, $(1,6,5)$, $(1,7,7)$, $(1,7,10)$, $(1,7,12)$, $(1,8,12)$, $(1,9,0)$, $(1,9,4)$, $(1,9,6)$, $(1,9,9)$, $(1,9,11)$, $(1,9,12)$, $(1,10,0)$, $(1,10,5)$, $(1,10,10)$, $(1,11,1)$, $(1,12,6)$, $(1,12,8)$, $(1,12,9)\}$.
}

\providecommand{\bysame}{\leavevmode\hbox to3em{\hrulefill}\thinspace}
\providecommand{\MR}{\relax\ifhmode\unskip\space\fi MR }
\providecommand{\MRhref}[2]{%
  \href{http://www.ams.org/mathscinet-getitem?mr=#1}{#2}
}
\providecommand{\href}[2]{#2}

\end{document}